\documentclass[12pt,leqno,oneside]{amsart}
\usepackage{mathrsfs,dsfont}
\usepackage{amsmath,amstext,amsthm,amssymb,bbm,comment}
\usepackage{charter}
\usepackage{typearea}
\usepackage{hyperref}
\usepackage{color}

\usepackage{float}
\usepackage{tikz}
\usetikzlibrary{patterns}
\usepackage{graphicx}
\usepackage{hyperref}

\usepackage[width=6.4in,height=8.5in]
{geometry}

\def\bt{\begin{thm}}
	\def\et{\end{thm}}
\def\bl{\begin{lem}}
	\def\el{\end{lem}}
\def\bd{\begin{defn}}
	\def\ed{\end{defn}}
\def\bc{\begin{cor}}
	\def\ec{\end{cor}}
\def\bp{\begin{proof}}
	\def\ep{\end{proof}}
\def\br{\begin{rem}}
	\def\er{\end{rem}}

\newtheorem{thm}{Theorem}[section]
\newtheorem{prop}[thm]{Proposition}
\newtheorem{lem}[thm]{Lemma}
\newtheorem{defn}[thm]{Definition}
\newtheorem{example}[thm]{Example}
\newtheorem{rem}[thm]{Remark}
\newtheorem{cor}[thm]{Corollary}

\numberwithin{equation}{section}

\title{An Exponential Rarefaction Result for\\ Sub-Gaussian Real Algebraic Maximal Curves}

\author{Turgay Bayraktar$^{1}$ and Emel Karaca$^{2}$}
\thanks{T. Bayraktar is partially supported by Turkish Academy of Sciences, GEBIP grant.}
\thanks{T. Bayraktar and E. Karaca are partially supported by T\"{U}B\.{I}TAK grant ARDEB-1001/119F184}
\address{$^1$Faculty of Engineering and Natural Sciences, Sabanc{\i} University, \.{I}stanbul, 34956 Turkey}
\email{tbayraktar@sabanciuniv.edu.tr}
\address{$^2$Polatl{\i} Faculty of Science and Arts, Ankara Hac{\i} Bayram Veli University, Ankara, 06900 Turkey}
\email{emel.karaca@hbv.edu.tr}

\date{\today}

\begin{document}
	
	\begin{abstract}
		We prove that maximal real algebraic curves associated with sub-Gaussian random real holomorphic sections of a smoothly curved ample line bundle are exponentially rare. This generalizes the result of Gayet and Welschinger \cite{GW} proved in the Gaussian case for positively curved real holomorphic line bundles.  		
	\end{abstract}
	\maketitle
	\section{Introduction}
	
	Let $X$ be a projective manifold of dimension $m$ and $L\to X$ be an ample line bundle. Assume that the real locus $\textbf{R}X$ is non-empty and $\pi:(L,c_L)\to(X,c_X)$ is a real holomorphic line bundle. The latter means that $\pi:L\to X$ is a holomorphic line bundle and the anti-holomorphic involutions $c_X$ and $c_L$ satisfy $\pi_*\circ c_L = c_X\circ \pi_*$ where $\pi_*$ denotes the differential of $\pi$. We denote the vector space of global holomorphic sections of the tensor power $L^{\otimes n}$ by $H^0(X,L^{\otimes n})$, its dimension is finite and will be denoted by $d_n$ (see eg. \cite{GH}). We also let $$\textbf{R} H^0(X,L^{\otimes n}):=\{s\in H^0(X,L^{\otimes n}): c_L\circ s=s\circ c_X\}$$ be the space of real sections whose real dimension is $d_n$. The discriminant locus (respectively, its real part) will be denoted by $\Delta_n\subset H^0(X,L^n)$ (respectively, $\textbf{R}\Delta_n$), it is the set of global holomorphic sections with singular zero locus. For $s\in H^0(X,L^n)$ we let $Z_s=s^{-1}(0)$ (respectively, $\textbf{R}Z_s$) be the complex (respectively real) zero locus of $s$. Note that for $s\in H^0(X,L^n)\setminus \Delta_n$ the hypersurface $Z_s$ is a complex manifold. For $s\in \textbf{R} H^0(X,L^{\otimes n})\setminus \textbf{R}\Delta_n$ we denote the number of connected components of the real zero locus by $b_0(\textbf{R}Z_s).$
	
Let $h$ be a singular Hermitian metric on $L$. Recall that for a holomorphic frame $e_L$ of $L$ on an open set $U\subset X$, we have $|e_L|_h=e^{-\varphi}$ for some $\varphi\in L^1_{loc}(U)$ (see e.g.\ \cite{Dem92}). The function $\varphi$ is called \textit{local weight function} of the metric $h$ with respect to the frame $e_L$. Moreover, the curvature current $c_1(L,h)_{|_U}=dd^c\varphi$ where $d=\partial + \bar{\partial}$ and $d^c=\frac{i}{2\pi}(\bar{\partial}-\partial)$. We say that a singular Hermitian metric is positive (respectively of class $\mathcal{C}^{\alpha}$) if the local weight functions are psh (respectively $\mathcal{C}^{\alpha}$) functions. Given a $\mathcal{C}^2$ Hermitian metric $h=e^{-\varphi}$ on $L$  one can define an \textit{extremal metric} $h_e:=e^{-\varphi_e}$ by taking the upper envelope of all local weight functions of positive metrics dominated by $\varphi$ (see \cite[\S3]{Ber}). More precisely, for the holomorphic frame $e_L$ on the open set $U$ we have
\begin{equation}\label{def}
\varphi_e:=\sup\{\psi\ \text{is a psh local weight}: \psi\leq \varphi\  \text{on}\ U\}.
\end{equation} 
It turns out that $h_e=e^{-\varphi_e}$ defines a singular Hermitian metric on $L$ such that its local weights are $\mathcal{C}^{1,1}$ psh functions. We denote its curvature current by $dd^c\varphi_e$; it is a globally well-defined positive closed $(1,1)$ current on $X$. By \cite{Ber}, the \textit{equilibrium measure} is defined by
	$$\mu_{\varphi_e}:=(dd^c\varphi_e)^m/m!$$  supported on the compact set
	$$S:=\overline{X_{h}^+\cap D}$$ where $D:=\{x\in X: \varphi(x)=\varphi_e(x)\}$ and the set $$X_{h}^+:=\{x\in X: dd^c\varphi(x)>0\}.$$ Note that the set $X_{h}^+$ depends only on the metric $h$ but not on the local weight $\varphi$. 
	
	The geometric data given above allow us to define a Hermitian inner product on the vector space of \textit{global holomorphic sections} $H^0(X,L^{\otimes n})$ via
	\begin{equation}\label{inp}\langle s_1,s_2\rangle:=\int_X \langle s_1(x),s_2(x)\rangle_{h^{\otimes n}} dV
	\end{equation} where $dV$ is a fixed volume form on $X$. The restriction of (\ref{inp}) to $\textbf{R} H^0(X,L^{\otimes n})$ induces an Euclidean inner product. We also denote the induced norm by $\|\cdot\|_n$. 
	
	Let $\{S_j^n\}_{j=1}^{d_n}$ be an  orthonormal basis for $\textbf{R} H^0(X,L^{\otimes n})$ with respect to the inner product (\ref{inp}). We consider \textit{sub-Gaussian random holomorphic sections} that are of the form
	$$S_n:=\sum_{j=1}^{d_n}c_j^nS_j^n$$ where $c_j^n$ are i.i.d.\ real sub-Gaussian random variables of mean zero and unit variance on some fixed probability space (see \S \ref{subG} for the definition and standart examples of sub-Gaussian random variables). This  induces a $d_n$-fold product probability measure $Prob_n$ on the vector space $\textbf{R}H^0(X,L^{\otimes n})$. We remark that unlike the Gaussian case, the probability structure in this more general setting depends on the choice of the orthonormal basis. Throughout this note, we assume that the probability distribution of $c_j$ is absolutely continuous with respect to Lebesgue measure and hence, with $Prob_n$ probability one the real zero loci $\textbf{R}Z_{s_n}$ is non-singular. We also consider the product probability space 
	\begin{equation}
	\prod_{n=1}^{\infty}\big(\textbf{R}H^0(X,L^{\otimes n}), Prob_n\big)
	\end{equation} whose elements are sequences of random real holomorphic sections. We refer the reader to the manuscripts \cite{Bay18,Bay20} for examples of local probabilistic models. 
	
	When $X$ is a projective surface, it follows from Harnack-Klein inequality that the number of real connected components of a non-singular real algebraic curve $b_0(\textbf{R}Z_s)$ is bounded by $g(Z_s)+1$. Here, $g(Z_s)$ denotes the genus of the Riemann surface $Z_s$ and it is given by the adjunction formula $g(Z_s)=\frac{1}{2} (n^{2}L^{2}-nc_{1}(X)L+2),$ where $c_{1}(X)$ is the first Chern class of $X.$ For $n \in \mathbb{N}\setminus \{0\}$ and $a >0$ we consider the set
	\begin{equation}
		M^{a}_{n}:=\{s \in {\bf{R}}H^0(X,L^n)\setminus {\bf{R}}\Delta_n : b_{0}(\textbf{R}Z_s) \geq g(Z_s) +1- an\}.
	\end{equation}
	The set $M^{a}_{n}$ is always non-empty for sufficiently large $n$ (see e.g.\ \cite{GW}). Gayet and Welschinger \cite[Theorem 1]{GW} proved that when $(L,h)$ is positively curved (i.e. the curvature form $c_1(L,h)>0$) and the random coefficients $c_j^n$ are i.i.d. standard Gaussians, the probability of the set $M^{a}_{n}$ decreases exponentially fast. More recently, Diatta and Lerario \cite{Di} (see also \cite{An}) generalized this result to higher dimensions by using a different approach: namely, with high probability for a Gaussian random real algebraic section one can approximate the real zero locus without changing its topology by real zero loci of a section associated with a lower tensor power. 
	
	In this note, we generalize \cite[Theorem 1]{GW} in three directions: we allow merely smooth Hermitian metric on the holomorphic line bundle $L$ and let the random coefficients $c_j^n$ be i.i.d. continuous sub-Gaussian random variables with mean zero and unit variance. Moreover, the exponential decay rate is far better. More precisely, we prove the following result: 
	
	\begin{thm}\label{main}
		Let $X$ be a real projective surface and $L\to X$ be a real holomorphic ample line bundle endowed with  a real $\mathcal{C}^2$ metric $h=e^{-\varphi}$. Assume that $X_{h}^+$ is non-empty. Then for each $a>0$ there exists constants $C > 0$ and $n_0\in\Bbb{N}$ such that
		\begin{equation}\label{claim1}
			Prob_n(M^{a}_{n}) \leq C n^2 e^{-n^2}
		\end{equation}
		for $n\geq n_0$.
	\end{thm}
	
	In the proof of Theorem \ref{main}, we rely on the observation by Gayet and Welschinger \cite{GW} that  for real holomorphic sections contained in $M^{a}_{n}$ if the corresponding normalized (complex) current of integrations along $Z_{s_n}$ converge to a positive closed (1,1) current $T$ then $T$ is weakly laminar outside the real locus $\textbf{R}X$ (Proposition \ref{theo}). In the current geometric setting, the distribution of the complex zeros is almost surely determined by the extremal current $dd^c\varphi_e$ (Theorem \ref{randommass}). Moreover, $dd^c\varphi_e$ is nowhere weakly laminar in the bulk (Proposition \ref{prop}). In order to estimate the probability of the exceptional set $M^{a}_{n}$ we use mass asymptotics (Theorem \ref{mass}) together with Hanson-Wright Inequality Theorem \ref{HWthm} which leads to a sharper exponential decay rate than the one obtained in \cite[Theorem 1]{GW}.

	\section{Preliminaries}
 
	\subsection{Laminar currents}
	We begin with some definitions and basic results on laminar currents that will be useful in the sequel. Since the definitions are local we fix a domain $\Omega$ in $\Bbb{C}^2$ and let $\mathcal{D}^{p,q}(\Omega)$ be the set of (smooth) test forms with compact support in $\Omega$. The dual space $\mathcal{D}_{p,q}(\Omega)$ is the set of $(p,q)$ currents. A $(1,1)$ current $T$ is called \textit{closed} if $\langle T,d\phi\rangle =0$ for every test 1-form. We say that $T$ is \textit{positive} if   $\langle T,\phi\rangle\geq 0$ for every positive $(1,1)$ test form. We refer the reader to the manuscript \cite{DemBook} for a detailed account of positive closed currents.  
	
	For a positive closed $(1,1)$ current $T$ we denote its (closed) support by $supp(T)$ and let $\|T\|$ be its trace measure and $M(T)$ be its mass norm. For example, if $C$ is a (possibly singular) holomorphic curve in $\Omega$ then the pairing 
	\begin{equation}\label{cint}
		\langle [C],\phi\rangle :=\int_{C_{reg}}\phi\ \ \text{for}\ \phi\in \mathcal{D}^{1,1}(\Omega)
	\end{equation} defines a positive closed $(1,1)$ current in $\Omega$ and called \textit{current of integration} along $C$.	
	
	\begin{defn} A $(1,1)$ current $T$ is called uniformly laminar if for each $x\in supp(T)$ there exists an open set $U$ containing $x$ and biholomorhic to the unit polydisc $\Bbb{D}^2$ and a measurable family of holomorphic functions $f_a:\Bbb{D}\to \Bbb{D}$ together with a measure $d\lambda$ on the set of parameters $\mathcal{A}:=\{f_a(0)\}$ such that the graphs $\Gamma_a=\{(y,f_a(y)):y\in \Bbb{D}\}$ are pairwise disjoint and we have
		$$\langle T, \phi\rangle =\int_{\mathcal{A}}(\int_{\Gamma_a}\phi)d\lambda(a) $$ for each $\phi\in \mathcal{D}^{1,1}(U)$.
	\end{defn}
	
	A uniformly laminar current is necessarily positive and closed. Clearly, a current of integration along a smooth algebraic curve of the form (\ref{cint}) is uniformly laminar. Demailly \cite{Dem82} proved that every positive closed current supported on a $\mathcal{C}^1$ Levi-flat hypersurface in $\Omega$ is uniformly laminar. 
	
	For two positive closed currents $T_1,T_2$ we may write $T_j=dd^cu_j$ for some psh functions on $\Omega$. We say that the exterior product $T_1\wedge T_2$ is \textit{admissible} if $u_1\in L^1_{loc}(\|T_2\|)$. Note that this condition is independent of the choice of the potential $u_1$ for $T_1$. In this case, the current $u_1T_2$ has locally bounded mass and the exterior product
	$$T_1\wedge T_2:= dd^c(u_1T_2)$$ is a positive measure with finite mass. We also remark that the exterior product is symmetric that is $u_1\in L^1_{loc}(\|T_2\|)$ if and only if $u_2\in L^1_{loc}(\|T_1\|)$ and 
	$$T_1\wedge T_2= T_2\wedge T_1. $$ 
	Let $S_j$ be positive closed currents satisfying $S_j\leq T_j$ (that is $T_j-S_j$ is positive) for $j=1,2$. It is well-know that  if $T_1\wedge T_2$ is admissible then $S_1\wedge S_2$ is also admissible and we have $$S_1\wedge S_2 \leq T_1\wedge T_2.$$ We will need the the following result in the sequel (see \cite[§6]{Du05}): 
	
	\begin{prop}\label{unif}
		Let $T=dd^cu$ for some psh $u$ in $\Omega$ for some $u\in L^{\infty}_{loc}(\Omega)$. Assume that $T$ is uniformly laminar. Then $T\wedge T=0$ on $\Omega$.
	\end{prop}	  
	
	Weakly laminar currents were introduced by Bedford-Lyubich-Smillie \cite{BLS} as a natural generalization of uniformly laminar currents. They arise naturally in pluri-potential theory and complex dynamics. We refer the reader to the papers \cite{BLS,Du} and references therein for more details.
	
	\begin{defn}
		A positive closed $(1,1)$ current $T$ in $\Omega$ is called weakly laminar if for each $\epsilon>0$ there exists a uniformly laminar current current $T_{\epsilon}$ on a subdomains $\Omega_{\epsilon}\subset \Omega$ such that $ T_{\epsilon}\leq T$ and $\|T_{\epsilon}\|(\partial \Omega_{\epsilon})=0$ satisfying $M((T-T_{\epsilon})\rvert_{\Omega_{\epsilon}})<\epsilon$.
	\end{defn}
	
	Equivalently (see \cite{BLS}), $T$ is weakly laminar if there exists a measurable family $(\mathcal{A},\lambda)$ of embedded holomorphic discs $D_a\subset \Omega$ such that for every pair $(a,b)$ the overlap $D_a\cap D_b$ is either empty or an open set in the disc topology and we have
	\begin{equation}
	T=\int_{\mathcal{A}}[D_a]d\lambda(a).
	\end{equation}
	
	The following are examples illustrating the difference between weakly laminar currents and uniformly laminar ones (see \cite{Dem82} and \cite[\S 5]{BLS} for more details):
	
	\begin{example}
		Let $(x,y)$ denote coordinates on $\Bbb{C}^2$ and 
		$$u_1(x,y)=\frac12\log^+(|x|^2+|y|^2)$$
		$$u_2(x,y)=\log^+(\max\{|x|,|y|\})$$ where $\log^+=\max(\log,0)$. The currents $dd^cu_1$ and $dd^cu_2$ are extremal and weakly laminar but they are not uniformly laminar. 
	\end{example}
Gayet and Welschinger \cite[Lemma 1]{GW} observed that a K\"ahler form does not have any laminar structure. We generalize this result to our geometric setting by proving that the extremal current associated with a smooth Hermitian metric does not have any laminar structure in the bulk: 	
	\begin{prop} \label{prop}
		Let $\varphi_e$ be the extremal metric defined by (\ref{def}). Then $dd^{c}\varphi_e$ is nowhere weakly laminar in the set $Int(D\cap X_{h}^+).$ 
	\end{prop}
	\begin{proof}
		Assume on the contrary that $T:=dd^c\varphi_e$ is weakly laminar near some point $z\in \Omega:=Int(D\cap X_{h}^+)$. Then we can find uniformly laminar currents $T_n\leq T$ on an increasing sequence of subdomains such that $z\in \Omega_n\subset \Omega$ and 
		\begin{equation}\label{mas} M((T-T_n)\rvert_{\Omega_n})<\frac1n.
		\end{equation} 
		Since $T_n\to T$ weakly, we may find negative psh potentials $u_n$ (resp. $u$) for $T_n$ (resp. $T$) such that $u_n\to u$ in $L^1_{loc}$. Since $T_n\leq T$ the function $u-u_n$ is psh and we may assume that $u\leq u_n\leq 0$. Thus, $u_nT_n$ has locally bounded mass. Moreover, by \cite[Proposition 3.2]{FS} we have 
		$$T_n\wedge T_n \to T\wedge T=\mu_{\varphi_e} $$ 
		weakly. On the other hand, by \cite[Theorem 3.4]{Ber} there exists $\delta>0$ such that $$\mu_{\varphi_e}(\Omega_n)=\int_{\Omega_n}\det(dd^c\varphi)dV>\delta$$ which contradicts Proposition \ref{unif}.
	\end{proof}
	
	The following result give sufficient conditions for weak laminarity:
	\begin{thm}\label{dTh}[de Th\'elin]
		Let $(C_{n})_{n \in \mathbf{N}}$ be a sequence of smooth holomorphic curves in the unit polydisc in $\Bbb{C}^2$  such that $g(C_{n})=O(A(C_{n})),$ where $g(C_{n})$ stands for the genus of $C_{n}$ and $A(C_{n})$ for its area. If $\frac{1}{A(C_{n})}[C_{n}]$ converges to a positive closed $(1,1)$ current $T$ as $n$ grows to infinity then $T$ is weakly laminar.
	\end{thm}

	\section{Quantum Ergodicity for Random Holomorphic Sections}
		\subsection{Subgaussian random variables}\label{subG}
	We recall some basic properties of subgaussian random variables. Assume that $(\Omega, \mathscr{F}, \tau )$ is a probability space. A real valued random variable $X: \Omega \rightarrow \mathbb{R}$
	is called \textit{subgaussian} with parameter $b>0$	if the moment generating function (MGF) of $X$ is dominated by MGF of a normalized Gaussian, that is there exists $b>0$ such that 
	\begin{equation*}
		\mathbb{E}[e^{tX}] \leq e^{\frac{b^{2}t^{2}}{2}}
	\end{equation*}
	for all $t \in \mathbb{R}.$ The classical examples of sub-Gaussian random variables are Standard Gaussian $N(0,1),$ Bernoulli random variables $\mathbb{P}[X=\pm 1]=\frac{1}{2},$ and uniform distribution on $[-1,1].$ We have the following characterizations for subgaussian random variables:
	\begin{prop} \cite[Lemma 5.5]{Ver} \label{ver}
		Let $X$ be a centered real random variable (i.e. $\mathbb{E}[X]=0$). Then the following equations are equivalent: \\
		(1) $\exists b>0$ such that $\mathbb{E}[e^{tX}]\leq e^{\frac{b^{2}t^{2}}{2}}$ for all $t \in \mathbb{R}.$ \\
		(2) $\exists c>0$ such that $\mathbb{P}[|X| > \alpha ]\leq 2e^{-c\alpha^{2}}$ for every $\alpha >0.$ \\
		(3) $\exists K>0$ such that $(\mathbb{E}(|X|^{p})^{\frac{1}{p}}) \leq K\sqrt{p}$ for all $p \geq 1.$ \\
		(4) $\exists \kappa >0$ such that $\mathbb{E}[e^{\frac{X^{2}}{\kappa^{2}}}] \leq 2.$
\end{prop}
The last property is known as \textit{$\psi_2$ condition}. Recall that for a centered random variable $X$  the \textit{Orlicz norm} is defined by 
\begin{equation} 
\|X\|_{\psi_{2}}:=\sup_{p\geq 1}p^{-1/2}(\Bbb{E}|X|^p)^{1/p}.
\end{equation} In particular, $X$ is subgaussian if and only if $\|X\|_{\psi_{2}}$ is finite.

\subsection{Hanson-Wright Inequality}
Let $X_{j}$ be independent subgaussian random variables and $\kappa_{j}:=\|X_{j}\|_{\psi_{2}}.$ The joint probability distribution of $X:=(X_{1},...,X_{N})$ is denoted by $\mathbb{P}.$ We also let  $A=[A_{ij}]$ be a square matrix with real entries. We denote its operator norm by
\begin{equation*}
\|A\|:=\max_{\|v\|_{2}\leq 1}\|Av\|.
\end{equation*}
where $\|.\|_{2}$ is the Euclidean norm. We also denote the Hilbert-Schmidt norm by
\begin{equation}\label{HSdef}
\|A\|_{HS}:=(\sum_{i,j} |a_{ij}|^{2})^{\frac{1}{2}}=[Tr(AA^{T})]^{\frac{1}{2}}.
\end{equation}
In \cite{HW}, Hanson-Wright  proved a concentration inequality for the values of a quadratic form
$$ X \rightarrow X^{T}AX$$ acting on a random vector $X$. The following version is due to Rudelson-Vershynin \cite{RV13}:
\begin{thm} (Hanson-Wright Inequality)\label{HWthm}
Assume that $A$ is a $N \times N$ square matrix and $X=(X_{1},...,X_{N}) \in \mathbb{R}^{N}$ is a random vector with components $X_{j}$ being independent subgaussian variables such that
\begin{equation*}
\|X_{j}\|_{\psi_{2}} \leq K
\end{equation*}
for $j=1,2,...,N.$ Then for each $t \geq 0$
\begin{equation*}
\mathbb{P}[|X^TAX-\mathbb{E}[X^TAX]|>t] \\
\leq2\exp\bigg(-c \min\{\frac{t^{2}}{K^{4}\left\| A\right\|_{HS}^{2}} \frac{t}{K^{2}\left\| A\right\|}\}\bigg)
\end{equation*}
where $c>0$ is an absolute constant which does not depend on $t.$
\end{thm}
	\subsection{Bergman Kernel Asymptotics}
	Let $X,L,e^{-\varphi}$ be as in the introduction and $\dim_{\Bbb{C}}X=m$. Recall that the \textit{Bergman kernel} for the $L^2$-space $(H^0(X,L^{\otimes n}), \langle ,\rangle)$ is the integral kernel of the orthogonal projection from $L^{2}$ space of global sections with values in $L^{\otimes n}$ onto $H^0(X,L^{\otimes n}).$ Note that the Bergman kernel satisfies the reproducing property
	\begin{equation*}
		s(x)=\int_{X}\langle s(x), K_{n}(x,y) \rangle_{h^{\otimes n}}dV(y)
	\end{equation*}
	for $s \in H^0(X,L^{\otimes n}).$ For any orthonormal basis $\{S^{n}_{j}\}$ for $H^0(X,L^{\otimes n})$ the Bergman kernel can be obtained as a smooth section
	\begin{equation*}
		K_{n}(x,y)= \sum_{j=1}^{d_{n}} S^{n}_{j}(x) \otimes \overline{S^{n}_{j}(y)}
	\end{equation*}
	of the line bundle $L^{\otimes n} \boxtimes (L^{\otimes n})^{*}$ over $X \times X.$ Note that this representation is independent of the choice of the orthonormal basis $S^{n}_{j}.$ The point-wise norm of the restriction of $K_{n}(x,y)$ to the diagonal is called the \textit{Bergman function}; it is given by
	\begin{equation*}
		k_{n}(x):=|K_{n}(x,x)|=\sum_{j=1}^{d_{n}}|S^{n}_{j}|^{2}_{h^{\otimes n}}.
	\end{equation*}
	We remark that $k_{n}(x)$ has the dimensional density property:
	\begin{equation*}
		\int_{X}k_{n}(x)dV=\dim(H^0(X,L^{\otimes n}))=d_n
	\end{equation*}
	for $n \in \mathbb{N}.$ 
	
	\begin{thm}\cite[Theorem 1.3]{Ber}\label{BerThm}
			Let $X$ be a real projective manifold and $L\to X$ be a real holomorphic line bundle endowed with a real $\mathscr{C}^2$ metric $h$. Then the following weak convergence of measures holds:
		\begin{equation*}
			n^{-m}k_{n}dV \rightarrow \mu_{\varphi_{e}},
		\end{equation*}
Moreover, 
		\begin{equation*}
			n^{-m}k_{n}(x) \rightarrow \textbf{1}_{D \cap X_{h}^+}det(dd^{c}\varphi)(x)
		\end{equation*}
		for almost every $x$ in $X,$ where $X_{h}^+$ is the set where $dd^{c}\varphi >0$.
	\end{thm}
	\subsection{Mass Asymptotics}
Let $X,L,h=e^{-\varphi}$ be as in the introduction. We consider the scalar $L^{2}$- product and the induced norm on the vector space of global holomorphic sections $H^0(X,L^{\otimes n})$ given by
\begin{equation*}
\|s\|^{2}_{n}:=\int_X |s|^{2}_{h^{\otimes n}} dV
\end{equation*} 
where $dV$ denotes the probability volume form induced by $w=c_1(L,h).$ When $(L,h)$ is positively curved, that is $dd^c\varphi>0$, Shiffman and Zelditch \cite{SZ} proved that for a sequence $(s_{n})_n \in H^0(X,L^{\otimes n})$ of global holomorphic sections if their \textit{masses} 
\begin{equation*}
\frac{1}{d_n}|s|^{2}_{h^{\otimes n}} dV \rightarrow dV
\end{equation*}
in the weak-star topology of measures on $X$ then the normalized current of integration $\frac{1}{n}[Z_{s_{n}}]$ also converge to the curvature form $c_1(L,h)$ in the sense of currents. More recently, Zelditch \cite{Zel} obtained a generalization of this result in the setting of partially positive metrics on positive line bundles. However, the proof of \cite[Theorem 1.2]{Zel} had a gap and it was fixed in \cite{Bay20}:
	
\begin{thm}\cite{Zel,Bay20}\label{mass}
		Let $X$ be a real projective manifold and $L\to X$ be a real holomorphic line bundle endowed with a real $\mathscr{C}^2$ metric $h$. Let $(s_n)_n$ be a sequence of global holomorphic sections. Assume that the masses \begin{equation}
			\frac{1}{d_n}|s_n(z)|^2_{h^{\otimes n}}dV \to d\mu_{\varphi_e}\ \text{as}\ n\to \infty
		\end{equation} in the weak-star sense on $S=\overline{X_{h}^+\cap D}.$ Then the normalized currents of integration
		$$\frac1n[Z_{s_n}] \to dd^c\varphi_e$$ in the sense of currents.
		\end{thm}	
	
	We also have a random version of Theorem \ref{mass} in setting of sub-Gaussian holomorphic sections:	
	\begin{thm}\label{randommass}\cite[Theorem 5.1]{Bay20}
		Let $X$ be a real projective manifold and $L\to X$ be a real holomorphic line bundle endowed with a real $\mathscr{C}^2$ metric $h$. Then for almost every sequence $(s_{n})_{n}$ in $\prod_{n=1}^{\infty}\big(\textbf{R}H^0(X,L^{\otimes n}), Prob_n\big)$ the masses 
		\begin{equation}
			\frac{1}{d_n}|s_n(z)|^2_{h^{\otimes n}}dV \to d\mu_{\varphi_e}
		\end{equation} in the weak-star sense on $S=\overline{X_{h}^+\cap D}.$
	\end{thm}

	\section{Proof of Theorem \ref{main}}
	The following result was obtained in \cite{GW}. We provide a proof for the sake of completeness.
	\begin{prop} \label{theo}
		Let $C_{n}$ be a sequence of smooth algebraic curves in $X$ of degree $\rho_n$ satisfying $$g(C_{n})+1-b_{0}(\textbf{R}C_{n})=O(\rho_n).$$ Assume that the normalized currents of integration $\frac{1}{\rho_n}[C_{n}]$ converge weakly to a positive closed $(1,1)$ current $T$ on $X$. Then $T$ is weakly laminar on $X\setminus \textbf{R}X$.
	\end{prop}
	
	\begin{proof}	By Theorem \ref{dTh}, it is enough to prove that for each ball $B \subset X\setminus \textbf{R}X$ we have $$g(C_{n}\cap B)=O(A_{n})$$ where $g(C_{n} \cap B)$ is the genus of the closed surface obtained by capping a disc at each boundary component of $C_{n} \cap B$ and $A_n=Area(C_{n} \cap B)$. To this end, first we observe that $g(C_{n}\backslash \textbf{R}C_{n})=O(\rho_n).$ Indeed, its Euler characteristic satisfies $$\chi(C_{n}\backslash  \textbf{R}C_{n})=\chi(C_{n})-\chi( \textbf{R}C_{n})=2-2g(C_{n}).$$ 
		The genus of the Riemann surface $g(C_{n}\backslash \ \textbf{R}C_{n})$  is given by the formula 
		$$\chi(C_{n}\backslash  \textbf{R}C_{n})=2b_{0}(C_{n}\backslash  \textbf{R}C_{n})-2g(C_{n}\backslash  \textbf{R}C_{n})-r(C_{n}\backslash \ \textbf{R}C_{n})$$ where $r(C_{n}\backslash \textbf{R}C_{n})=2b_{0}(\textbf{R}C_{n})$ is the number of ends of Riemann surface $C_{n}\backslash\textbf{R}C_{n}.$  By the argument on \cite[page 66 ]{Wilson} we have $b_{0}(C_{n}\backslash \textbf{R}C_{n})\leq 2$ which implies that
		\begin{eqnarray*}
			g(C_{n}\backslash \textbf{R}C_{n}) &=& b_{0}(C_{n}\backslash \textbf{R}C_{n})-\frac12\chi(C_{n}\backslash \textbf{R}C_{n})-\frac12r(C_{n}\backslash \textbf{R}C_{n})\\  &\leq& 1+g(C_{n})-b_{0}(\textbf{R}C_{n})\\ &=& O(\rho_n).
		\end{eqnarray*} 
		Now, the Euler characteristic $$\chi(C_{n} \cap B)=2b_{0}(C_{n} \cap B)-2g(C_{n} \cap B)-r(C_{n} \cap B).$$  Since the genus of the orientable surface can only increase under connected sums, we see that $g(C_{n} \cap B) \leq g(C_{n}\backslash \textbf{R}C_{n})=O(\rho_n).$ Finally, since $\frac{1}{\rho_n} [C_n] \to T$ we see that for sufficiently large $n$ we have $A(C_n\cap B)\geq c \rho_n$ for some $c>0$ depending on $T$ and $B$. Thus, the assertion follows from Theorem \ref{dTh}.
	\end{proof}

	\begin{proof}[Proof of Theorem \ref{main}]
First, we observe that $$Prob_n[s_n\in \textbf{R}H^0(X,L^{\otimes n}):\|s_n\|_n>d_n]\leq d_n\exp(-d_n).$$ Indeed, for $s_n:=\sum_{j=1}^{d_n}c_j^nS_j^n$ we have
		\begin{eqnarray} \label{eq}
			Prob_n[\|s_n\|_n>d_n] &=& Prob_n[ \sum_{j=1}^{d_n} |c_j^n|^{2} >d_n^2] \nonumber \\
			& \leq & Prob_{n} [|c_j^n|^{2} >d_{n} \ \ for \ \ some \ \ j] \nonumber \\
			& \leq & 2d_{n} \exp({-d_{n}})=O(d_{n} \exp({-d_{n}})).
		\end{eqnarray}
		Let $R_{n}:=M_n^{a} \cap \{ \|s_n\|_{n} \leq d_n  \}$
		and  $\mathcal{M}(X)$ be the set of Borel measures of total mass at most one. We also let $$ \Omega_{n}:=\{ \nu_{s_n}:\ \text{for some}\ s_n\in R_n \}\subset \mathcal{M}(X)$$ where $\nu_{s_n}:=\frac{1}{d_n}|s_n|^2_{h^{\otimes n}}dV$ denotes the mass of $s_n$. Note that in order to prove (\ref{claim1}), by (\ref{eq}), it is enough to show that $$Prob_{n} ( R_{n}) =O(\exp(-d_n)).$$ 
		To this end, let us consider the compact set $\Omega:= \overline {\cup \Omega_{n}}\subset \mathcal{M}(X)$ where the closure is taken with respect to the weak-star topology. Note that by Proposition \ref{prop} and Theorem \ref{mass} we have $\mu_{\phi_{e}} \notin \Omega.$ Hence, by compactness there exists an $\epsilon>0$ and finitely many test functions $\phi_{j}$ on $S_{\phi} \backslash  \textbf{R} X$ such that 
		for each $\nu_{s_n} \in \Omega$ we have 
		\begin{equation}
			| \frac{1}{d_{n}}X^{\phi_{j}}_{n}(s_{n})-\int_X \phi_{j}d\mu_{\varphi_e}|> \epsilon 
		\end{equation} 
		where
		\begin{equation*}
			X^{\phi_{j}}_{n}(s_{n})= \int_X \phi_{j}|s_n(z)|^2_{h^{\otimes n}}dV. 
		\end{equation*}
		Next, by identifying $s_n\in {\bf{R}}H^0(X,L^n)$ with a vector in $\Bbb{R}^{d_n}$ and using linear algebra one can find an $d_n \times d_n$ symmetric matrix $A^{\phi_{j}}_{n}$ such that
		\begin{equation*}
			X^{\phi_{j}}_{n}(s_{n})= \langle A^{\phi_{j}}_{n}s_{n},s_{n}\rangle=:s^{T}_{n}A^{\phi_{j}}_{n}s_{n}. 
		\end{equation*} 
		where $\langle , \rangle$ is the standard inner product on $\mathbb{R}^{d_n}.$ Observe that $A^{\phi_{j}}_{n}$ is nothing but the matrix representing the Toeplitz operator $T^{\phi_{j}}_{n}:=\pi_{n}M_{\phi_{j}},$ where $\pi$ is the Bergman projection and $M_{\phi_{j}}$ is the multiplication operator induced by $\phi_{j}.$\\
		 By \cite[Proposition 3.8]{Bay20} we have 
		$$\mathbb{E}[X^{\phi_{j}}_{n}]=\int_{X} \phi_{j}k_{n}(z)dV=Tr(T^{\phi_{j}}_{n}).$$
		Then by Theorem \ref{HWthm} there exists an absolute constant $C>0$ such that
		\begin{eqnarray}\label{HW}
			Prob_{n}[|X^{\phi_{j}}_{n}(s_{n})-\mathbb{E}[X^{\phi_{j}}_{n}]|>t ]&=&\mathbb{P}_{n}[|s^{T}_{n}A^{\phi_{j}}_{n}s_{n}-\mathbb{E}[s^{T}_{n}A^{\phi_{j}}_{n}s_{n}]|>t] \nonumber\\
			&\leq& 2\exp\bigg(-C \min\{\frac{t^{2}}{K^{4}\left\| A^{\phi_{j}}_{n}\right\|_{HS}^{2}}, \frac{t}{K^{2}\left\| A^{\phi_{j}}_{n}\right\|}\}\bigg),
		\end{eqnarray}
		where $K:=\left\| c^{n}_{j}\right\|_{\psi_{2}} \geq \frac{1}{\sqrt{2}}.$ Note that the operator norm satisfies $\left\| A^{\phi_{j}}_{n}\right\| \leq \sup_{ z \in X}|\phi_{j}(z)|$. Moreover, since $A^{\phi_{j}}_{n}$ is symmetric by (\ref{HSdef}) and \cite[Proposition 3.3]{Bay20} we have
		\begin{equation*}
			\left\| A^{\phi_{j}}_{n}\right\|_{HS}^{2}=Tr((T^{\phi_{j}}_{n})^{2})=O(d_{n}).
		\end{equation*}
		Furthermore, by Theorem \ref{BerThm} and using $d_n=O(n^2)$ we have
		$$|\int_X \phi_{j}d\mu_{\varphi_e}- \frac{1}{n^2}\mathbb{E}[X^{\phi_{j}}_{n}]| < \frac{\varepsilon}{2} $$ for all $j$ and sufficiently large $n$. Then by applying triangle inequality first and letting $t=\epsilon d_{n}$ in (\ref{HW}) we obtain
		\begin{eqnarray} \label{eqn}
			Prob_{n}[| \frac{1}{d_{n}}X^{\phi_{j}}_{n}(s_{n})-\int_X \phi_{j}d\mu_{\varphi_e}|> \varepsilon]   
			&\leq& Prob_{n}[| \frac{1}{d_{n}}X^{\phi_{j}}_{n}(s_{n})- \frac{1}{n^2}\mathbb{E}[X^{\phi_{j}}_{n}]|> \varepsilon/2] \nonumber  \\
			&= &O (\exp({-n^2}))
		\end{eqnarray}
		Thus, combining (\ref{eq}) and (\ref{eqn}) we deduce that
		$$Prob_n[M^{a}_n]=O(n^2 \exp(-n^2)). $$
	\end{proof}

\end{document}